\documentclass{amsart}
\usepackage{amscd}
\usepackage{amsmath,empheq}
\usepackage{amsfonts}
\usepackage{amssymb}
\usepackage{mathrsfs}

\usepackage[all]{xy}

\newtheorem{theorem}{Theorem}

\newtheorem{prop}[theorem]{Proposition}
\newtheorem{remark}{Remark}

\newtheorem{claim}{Claim}

\newenvironment{proof-sketch}{\noindent{\bf Sketch of Proof}\hspace*{1em}}{\qed\bigskip}

\everymath{\displaystyle}

\newcommand{\RR}{\mathbb R}
\newcommand{\NN}{\mathbb N}

\newcommand{\ZZ}{\mathbb Z}

\renewcommand{\leq}{\leqslant}

\renewcommand{\geq}{\geqslant}

\begin{document}

\title[Multiple solutions for resonant problems]{Multiple solutions for resonant problems of the Robin $p$-Laplacian plus an indefinite potential}

\author[N.S. Papageorgiou]{Nikolaos S. Papageorgiou}
\address{National Technical University, Department of Mathematics,
				Zografou Campus, Athens 15780, Greece}
\email{\tt npapg@math.ntua.gr}

\author[V.D. R\u{a}dulescu]{Vicen\c{t}iu D. R\u{a}dulescu}
\address{Department of Mathematics, Faculty of Sciences, King Abdulaziz University, P.O. Box 80203, Jeddah 21589, Saudi Arabia \& Department of Mathematics, University of Craiova, 200585 Craiova, Romania}
\email{\tt vicentiu.radulescu@imar.ro}

\author[D.D. Repov\v{s}]{Du\v{s}an D. Repov\v{s}}
\address{Faculty of Education and Faculty of Mathematics and Physics, University of Ljubljana,
SI-1000 Ljubljana, Slovenia}
\email{\tt dusan.repovs@guest.arnes.si}


\begin{abstract}
We study  a nonlinear boundary value problem driven by the $p$-Laplacian plus an indefinite potential with Robin boundary condition. The reaction term is a Carath\'eodory function which is asymptotically resonant at $\pm\infty$ with respect to a nonprincipal Ljusternik-Schnirelmann eigenvalue. Using variational methods, together with Morse theory and truncation-perturbation techniques, we show that the problem has at least three nontrivial smooth solutions, two of which have a fixed sign.
\end{abstract}

\maketitle


\section{Introduction}

Let $\Omega\subseteq\RR^N$ be a bounded domain with a $C^2$-boundary $\partial\Omega$. In this paper we study the following nonlinear Robin boundary value problem
\begin{equation}\label{eq1}
	\left\{\begin{array}{ll}
		-\Delta_pu(z)+\xi(z)|u(z)|^{p-2}u(z)=f(z,u(z))&\mbox{in}\ \Omega,\\
		\frac{\partial u}{\partial n_p}+\beta(z)|u|^{p-2}u=0&\mbox{on}\ \partial\Omega.
	\end{array}\right.
\end{equation}

Here $\Delta_p$ denotes the $p$-Laplacian differential operator defined by
$$\Delta_pu={\rm div}\,(|Du|^{p-2}Du)\ \mbox{for all}\ u\in W^{1,p}(\Omega),1<p<\infty.$$

Also, $\xi\in L^{\infty}(\Omega)$ is an indefinite (that is, sign changing) potential function. The reaction term $f(z,x)$ is a Carath\'eodory function (that is, for all $x\in\RR$ the mapping $z\mapsto f(z,x)$ is measurable and for almost all $z\in\Omega$ the function $x\mapsto f(z,x)$ is continuous) which is $p$-sublinear in $x\in\RR$ and asymptotically interacts as $x\rightarrow\pm\infty$ with the nonprincipal part of the spectrum of $W^{1,p}(\Omega)\ni u\mapsto -\Delta_pu+\xi(z)|u|^{p-2}u$ with Robin boundary condition. In the boundary condition, $\frac{\partial u}{\partial n_p}$ denotes the generalized normal derivative on $\partial\Omega$, defined by
$$\frac{\partial u}{\partial n_p}=|Du|^{p-2}(Du,n)_{\RR^N}=|Du|^{p-2}\frac{\partial u}{\partial n}\ \mbox{for all}\ u\in W^{1,p}(\Omega),$$
with $n(\cdot)$ being the outward unit normal on $\partial\Omega$. The boundary coefficient function $\beta(\cdot)$ satisfies $\beta\in C^{0,\alpha}(\partial\Omega)$ with $0<\alpha<1$ and $\beta(z)\geq 0$ for all $z\in \partial\Omega$. When $\beta\equiv 0$, we get the Neumann problem. In this resonant setting, we prove a multiplicity theorem, producing three nontrivial smooth solutions.

Previous three solutions theorems for equations driven by the $p$-Laplacian with zero potential function (that is, $\xi\equiv 0$), were proved by Gasinski \& Papageorgiou \cite{8}, Guo \& Liu \cite{10}, Liu \& Liu \cite{13}, Motreanu, Motreanu \& Papageorgiou \cite{15}, Papageorgiou \& Papageorgiou \cite{18}, Papageorgiou \& R\u{a}dulescu \cite{19,20} (Dirichlet problems) and Gasinski \& Papageorgiou \cite{9}, Motreanu, Motreanu \& Papageorgiou \cite{14} (Neumann problems). Recently, Mugnai \& Papageorgiou \cite{17} considered Neumann problems driven by the $p$-Laplacian plus an indefinite potential, while Papageorgiou \& R\u{a}dulescu \cite{21, 22} studied Robin problems driven by the $p$-Laplacian with no potential function. Of the aforementioned works, resonant problems are treated only in \cite{13}, \cite{15}, \cite{17}, and \cite{21}. In all these works resonance occurs with respect to the principal eigenvalue. The reason for this is that due to the nonlinearity of the
 differential operator, the eigenspaces are not linear spaces and consequently the underlying Sobolev space does not have a direct sum decomposition in terms of the eigenspaces. In addition, we have only partial knowledge of the spectrum of the operator. These drawbacks make the study of problems which are resonant with respect to the nonprincipal part of the spectrum, rather difficult and require different tools to prove existence and multiplicity theorems. Our approach uses variational methods based on the critical point theory for linking sets over cones (since the eigenspaces are cones), together with Morse theory (critical groups). With these tools we prove the existence of at least three nontrivial smooth solutions.

\section{Mathematical Background}

In this section, for the convenience of the reader, we briefly recall the main mathematical tools which will be used in the sequel.

So, let $X$ be a Banach space with $X^*$ its topological dual. By $\left\langle \cdot,\cdot\right\rangle$ we denote the duality brackets for the pair $(X^*,X)$. Given $\varphi\in C^1(X,\RR)$, we say that $\varphi$ satisfies the ``Cerami condition'' (the C-condition for short), if the following property holds:
\begin{center}
``Every sequence $\{u_n\}_{n\geq 1}\subseteq X$ such that $\{\varphi(u_n)\}_{n\geq 1}\subseteq\RR$ is bounded and
$$(1+||u_n||)\varphi'(u_n)\rightarrow 0\ \mbox{in}\ X^*,$$
admits a strongly convergent subsequence''.
\end{center}

Since the underlying Banach space is in general not locally compact (in most applications $X$ is infinite dimensional), the necessary compactness condition needed to have a coherent theory is passed to the functional $\varphi$ by introducing the aforementioned C-condition. This is analogous to what happens in infinite dimensional degree theory (the Leray-Schauder degree theory). The C-condition leads to a deformation theorem describing the changes in the topological structure of the sublevel sets $\varphi$. From this, one can derive the minimax theory of the critical values of $\varphi$. Prominent in that theory, is the so-called ``mountain pass theorem'' due to Ambrosetti and Rabinowitz \cite{2}.
\begin{theorem}\label{th1}
	Suppose that $\varphi\in C^1(X,\RR)$ satisfies the C-condition, $u_0,u_1\in X$, $||u_1-u_0||>\rho>0$
	$$\max\{\varphi(u_0),\varphi(u_1)\}<\inf[\varphi(u):||u-u_0||=\rho]=m_{\rho}$$
	and $c=\inf\limits_{\gamma\in\Gamma}\max\limits_{0\leq t\leq 1}\varphi(\gamma(t))$, where $\Gamma=\{\gamma\in C([0,1],X):\gamma(0)=u_0,\gamma(1)=u_1\}$. Then $c\geq m_{\rho}$ and $c$ is a critical value of $\varphi$.
\end{theorem}

For our problem, the underlying space is the Sobolev space $W^{1,p}(\Omega)$ with the norm
$$||u||=\left[||u||^p_p+||Du||^p_p\right]^{1/p}\ \mbox{for all}\ u\in W^{1,p}(\Omega).$$

In addition, we will also use the ordered Banach space $C^1(\overline{\Omega})$, with positive (order) cone given by
$$C_+=\{u\in C^1(\overline{\Omega}):u(z)\geq 0\ \mbox{for all}\ z\in\overline\Omega\}.$$

In $C_+$ we consider the nonempty open set $D_+$ given by
$$D_+=\{u\in C_+:u(z)>0\ \mbox{for all}\ z\in\overline{\Omega}\}.$$

To deal with the Robin boundary condition, we will need the ``boundary'' Lebesgue spaces $L^q(\partial\Omega)\ (1\leq q\leq\infty)$. So, on $\partial\Omega$ we consider the $(N-1)$-dimensional Hausdorff (surface) measure denoted by $\sigma(\cdot)$. Using this measure, we can define the Lebesgue spaces $L^q(\partial\Omega)$. From the theory of Sobolev spaces, we know there exists a unique linear continuous map $\gamma_0:W^{1,p}(\Omega)\rightarrow L^q(\partial\Omega)$, known as the trace map, such that $\gamma_0(u)=u|_{\partial\Omega}$ for all $u\in W^{1,p}(\Omega)\cap C(\overline{\Omega})$. So, we can understand the trace map as representing the boundary values of $u\in W^{1,p}(\Omega)$. We know that $\gamma_0$ is a compact mapping into $L^q(\partial\Omega)$ for all $q\in\left[1,\frac{p(N-1)}{N-p}\right)$ and
$${\rm im}\,\gamma_0=W^{\frac{1}{p'},p}(\partial\Omega)\left(\frac{1}{p'}+\frac{1}{p}=1\right)\ \mbox{and}\ {\rm ker}\,\gamma_0=W^{1,p}_0(\Omega).$$

In the sequel, for the sake of notational simplicity, we shall drop the use of the map $\gamma_0$. It is understood that all restrictions of the Sobolev functions $u\in W^{1,p}(\Omega)$ on $\partial\Omega$, are defined in the sense of traces.

Next, we recall some basic facts about the spectrum of the differential operator $u\mapsto -\Delta_pu+\xi(z)|u|^{p-2}u$ with Robin boundary condition. For details we refer to Mugnai \& Papageorgiou \cite{17} and Papageorgiou \& R\u{a}dulescu \cite{21}.

We consider the following nonlinear eigenvalue problem:
\begin{equation}\label{eq2}
	\left\{\begin{array}{ll}
		-\Delta_pu(z)+\xi(z)|u(z)|^{p-2}u(z)=\hat{\lambda}|u(z)|^{p-2}u(z)&\mbox{in}\ \Omega,\\
		\frac{\partial u}{\partial n_p}+\beta(z)|u|^{p-2}u=0&\mbox{on}\ \partial\Omega.
	\end{array}\right.
\end{equation}

We say that $\hat{\lambda}\in\RR$ is an eigenvalue of the differential operator, if problem (\ref{eq2}) admits a nontrivial solution $\hat{u}\in W^{1,p}(\Omega)$, known as an eigenfunction corresponding to the eigenvalue $\hat{\lambda}$. We know that there exists a smallest eigenvalue $\hat{\lambda}_1\in\RR$ which has the following properties:
\begin{itemize}
	\item $\hat{\lambda}_1$ is isolated in the spectrum $\hat{\sigma}(p)$ of the differential operator;
	\item $\hat{\lambda}_1$ is simple (that is, if $\hat{u}_1,\hat{u}_2$ are eigenfunctions corresponding to $\hat{\lambda}_1$, then $\hat{u}_1=\xi\hat{u}_2$ for some $\xi\neq 0$); and
	\item $\hat{\lambda}_1=\inf\left[\frac{\gamma(u)}{||u||^p_p}:u\in W^{1,p}(\Omega),u\neq 0\right]$, where $\gamma:W^{1,p}(\Omega)\rightarrow\RR$ is the $C^1$-functional defined by
	\begin{equation}\label{eq3}
		\gamma(u)=||Du||^p_p+\int_{\Omega}\xi(z)|u|^pdz+\int_{\partial\Omega}\beta(z)|u|^pd\sigma\ \mbox{for all}\ u\in W^{1,p}(\Omega).
	\end{equation}
\end{itemize}

The infimum in (\ref{eq3}) is realized on the corresponding one dimensional eigenspace. Evidently, the elements of this eigenspace do not change sign. In what follows, by $\hat{u}_1$ we denote the positive $L^p$-normalized (that is, $||\hat{u}_1||_p=1$) eigenfunction corresponding to the eigenvalue $\hat{\lambda}_1$. The nonlinear regularity theory (see Lieberman \cite[Theorem 2]{12}), implies that $\hat{u}_1\in C_+\backslash\{0\}$. Moreover, using the nonlinear strong maximum principle (see, for example, Gasinski \& Papageorgiou \cite[p. 738]{7}), we show that $\hat{u}_1\in D_+$.

It is easy to show that the spectrum $\hat{\sigma}(p)$ is closed and so given the isolation of $\hat{\lambda}_1$, the second eigenvalue of (\ref{eq2}) is well-defined by
$$\hat{\lambda}_2=\inf[\hat{\lambda}\in\hat{\sigma}(p):\hat{\lambda}>\hat{\lambda}_1].$$

Using the Ljusternik-Schnirelmann minimax scheme with the Fadell-Rabinowitz  cohomological index {\rm ind}$\,(\cdot)$ (see \cite{6}), we produce a whole sequence $\{\hat{\lambda}_k\}_{k\geq 1}$ of distinct eigenvalues such that $\hat{\lambda}_k\rightarrow+\infty$ (see Cingolani \& Degiovanni \cite{4}). However, we do not know if this sequence exhausts $\hat{\sigma}(p)$. This is the case if $p=2$ (linear eigenvalue problem) or $N=1$ (ordinary differential equation). All eigenvalues $\hat{\lambda}\neq\hat{\lambda}_1$ have nodal (that is, sign changing) eigenfunctions.

As we already mentioned in the introduction, our approach involves also the use of critical groups (Morse theory). So, let us briefly recall some basic definitions and facts from that theory.

So, let $X$ be a Banach space, $\varphi\in C^1(X,\RR)$ and $c\in\RR$. We introduce the following sets:
\begin{eqnarray*}
	&&K_{\varphi}=\{u\in X:\varphi'(u)=0\}\ (\mbox{the critical set of}\ \varphi),\\
	&&K_{\varphi}^c=\{u\in K_{\varphi}:\varphi(u)=c\}\ (\mbox{the critical set of}\ \varphi\ \mbox{at the level}\ c),\\
	&&\varphi^c=\{u\in X:\varphi(u)\leq c\}\ (\mbox{the sublevel set of}\ \varphi\ \mbox{at}\ c).
\end{eqnarray*}

Let $(Y_1,Y_2)$ be a topological pair such that $Y_2\subseteq Y_1\subseteq X$. For every $k\in\NN_0$, by $H_k(Y_1,Y_2)$ we denote the $k$th-relative singular homology group with integer coefficients for the pair $(Y_1,Y_2)$. Recall that $H_k(Y_1,Y_2)=0$ for all $k\in-\NN$. If $u\in K^c_{\varphi}$ is isolated, then the critical groups of $\varphi$ at $u$ are defined by
$$C_k(\varphi,u)=H_k(\varphi^c\cap U,\varphi^c\cap U\backslash\{u\})\ \mbox{for all}\ k\in\NN_0.$$

Here, $U$ is a neighborhood of $u$ such that $K_{\varphi}\cap \varphi^c\cap U=\{u\}$. The excision property of the singular homology implies that the above definition of critical groups is independent of the particular choice of the neighborhood $U$.

If $\varphi\in C^1(X,\RR)$ satisfies the C-condition and $\inf\varphi(K_{\varphi})>-\infty$, then the critical groups of $\varphi$ at infinity are defined by
$$C_k(\varphi,\infty)=H_k(X,\varphi^c)\ \mbox{for all}\ k\in\NN_0,$$
with $c<\inf\varphi(K_{\varphi})$. The second deformation theorem (see Gasinski \& Papageorgiou \cite[p. 628]{7}), implies that this definition is independent of the choice of the level $c<\inf\varphi(K_{\varphi})$.

Assuming that $K_{\varphi}$ is infinite, we define
\begin{eqnarray*}
	&&M(t,u)=\sum_{k\geq 0}{\rm rank}\, C_k(\varphi,u)\ \mbox{for all}\ t\in\RR,\ \mbox{all}\ u\in K_{\varphi},\\
	&&P(t,\infty)=\sum_{k\geq 0}{\rm rank}\, C_k(\varphi,\infty)\ \mbox{for all}\ t\in\RR.
\end{eqnarray*}

The Morse relation says that
\begin{equation}\label{eq4}
 {\underset{\mathrm{u\in K_{\varphi}}}\sum}M(t,u)=P(t,\infty)+(1+t)Q(t)\ \mbox{for all}\ t\in\RR,
\end{equation}
where $Q(t)={\underset{\mathrm{k\leq 0}}\sum}\beta_kt^k$ is a formal series in $t\in\RR$ with nonnegative integer coefficients $\beta_k$.

From (\ref{eq4}) it follows easily that, if for some $m\in\NN$ we have $C_m(\varphi,\infty)\neq 0$, then there exists $u\in K_{\varphi}$ such that $C_m(\varphi,u)\neq 0$. Moreover, if $u\in X$ is a local minimizer of $\varphi$, then
$$C_k(\varphi,u)=\delta_{k,0}\ZZ\ \mbox{for all}\ k\in\NN,$$
with $\delta_{k,0}$ being the Kronecker symbol, that is,
$$\delta_{k,\vartheta}=\left\{\begin{array}{ll}
	1&\mbox{if}\ k=\vartheta\\
	0&\mbox{if}\ k\neq \vartheta
\end{array}\right.\ \mbox{for all}\ k,\vartheta\in\NN_0.$$

Let $A:W^{1,p}(\Omega)\rightarrow W^{1,p}(\Omega)^*$ be the nonlinear map defined by
$$\left\langle A(u),h\right\rangle=\int_{\Omega}|Du|^{p-2}(Du,Dh)_{\RR^N}dz\ \mbox{for all}\ u,h\in W^{1,p}(\Omega).$$

From Motreanu, Motreanu \& Papageorgiou \cite[p. 40]{16}, we have the following property.
\begin{prop}\label{prop2}
	The map $A:W^{1,p}(\Omega)\rightarrow W^{1,p}(\Omega)^*$ is bounded (maps bounded sets into bounded sets), demicontinuous, maximal monotone and of type $(S)_+$, that is,
	$$``u_n\stackrel{w}{\rightarrow}u\ \mbox{in}\ W^{1,p}(\Omega),\limsup\limits_{n\rightarrow\infty}\left\langle A(u_n),u_n-u\right\rangle\leq 0\Rightarrow u_n\rightarrow u\ \mbox{in}\ W^{1,p}(\Omega)".$$
\end{prop}

Also, as a consequence of the properties of $\hat{\lambda}_1$ and $\hat{u}_1$, we have (see Papageorgiou \& R\u{a}dulescu \cite{21}).
\begin{prop}\label{prop3}
	If $\vartheta\in L^{\infty}(\Omega),\ \vartheta(z)\leq\hat{\lambda}_1$ for almost everywhere $z$ in $\Omega$, $\vartheta\not\equiv\hat{\lambda}_1$, then there exists $c_0>0$ such that
	$$\gamma(u)-\int_{\Omega}\vartheta(z)|u|^pdz\geq c_0||u||^p\ \mbox{for all}\ u\in W^{1,p}(\Omega).$$
\end{prop}

We now  fix our notation. Given $x\in\RR$, we set $x^{\pm}=\max\{\pm x,0\}$. Then for $u\in W^{1,p}(\Omega)$ we define
$$u^{\pm}(\cdot)=u(\cdot)^{\pm}.$$

We have $u^{\pm}\in W^{1,p}(\Omega),\ u=u^+-u^-,\ |u|=u^++u^-$. Given a measurable function $g:\Omega\times\RR\rightarrow\RR$ (for example, a Carath\'eodory function), we define
$$N_g(u)(\cdot)=g(\cdot,u(\cdot))\ \mbox{for all}\ u\in W^{1,p}(\Omega),$$
which is the Nemytskii (superposition) operator corresponding to $g$.

Finally, we introduce the hypotheses on the data of problem (\ref{eq1}):

\smallskip
$H(\xi):\ \xi\in L^{\infty}(\Omega)$.

\smallskip
$H(\beta):\ \beta\in C^{0,\alpha}(\partial\Omega)$ with $\alpha\in(0,1)$ and $\beta(z)\geq 0$ for all $z\in\partial\Omega$.
\begin{remark}
	If $\beta\equiv 0$, then we obtain the Neumann problem. So, our framework here includes the Neumann problem as a special case.
\end{remark}

$H(f):\ f:\Omega\times\RR\rightarrow\RR$ is a Carath\'eodory function such that $f(z,0)=0$ for almost all $z\in\Omega$ and
\begin{itemize}
	\item[(i)] for every $\rho>0$, there exists $a_{\rho}\in L^{\infty}(\Omega)$ such that
	$$|f(z,x)|\leq a_{\rho}(z)\ \mbox{for almost all}\ z\in\Omega,\ \mbox{all}\ |x|\leq\rho;$$
	\item[(ii)] there exists an integer $m\geq 2$ such that
	$$\lim\limits_{x\rightarrow\pm\infty}\frac{f(z,x)}{|x|^{p-2}x}=\hat{\lambda}_m\ \mbox{uniformly for almost all}\ z\in\Omega;$$
	\item[(iii)] if $F(z,x)=\int^x_0f(z,s)ds$, then
	$$\lim\limits_{x\rightarrow\pm\infty}[f(z,x)x-pF(z,x)]=+\infty\ \mbox{uniformly for almost all}\ z\in\Omega; \ \mbox{and}$$
	\item[(iv)] there exist functions $\hat{\vartheta},\vartheta\in L^{\infty}(\Omega)$ such that
	$$\vartheta(z)\leq\hat{\lambda}_1\ \mbox{for almost all}\ z\in\Omega,\ \vartheta\not\equiv\hat{\lambda}_1,$$
	$\hat{\vartheta}(z)\leq\liminf\limits_{x\rightarrow 0}\frac{f(z,x)}{|x|^{p-2}x}\leq\limsup\limits_{x\rightarrow 0}\frac{f(z,x)}{|x|^{p-2}x}\leq\vartheta(z)$ uniformly for almost all $z\in\Omega$.
\end{itemize}
\begin{remark}
	Hypothesis $H(ii)$ implies that asymptotically at $\pm\infty$, we have resonance with respect to any nonprincipal Ljusternik-Schnirelmann eigenvalue of the differential operator. Hypotheses $H(f)(i),(ii)$ imply that
	\begin{equation}\label{eq5}
		|f(z,x)|\leq c_1(1+|x|^{p-1})\ \mbox{for almost all}\ z\in\Omega,\ \mbox{all}\ x\in\RR,\ \mbox{some}\ c_1>0.
	\end{equation}
\end{remark}

Hypothesis $H(f)(iv)$ says that  we have nonuniform nonresonance at zero with respect to the principal eigenvalue $\hat{\lambda}_1$ from the left.

\section{Constant Sign Solutions}

In this section, using minimax methods, we establish the existence of at least two nontrivial smooth solutions of constant sign (one positive and one negative).

Let $\varphi: W^{1,p}(\Omega)\rightarrow\RR$ be the energy (Euler) functional for problem (\ref{eq1}) defined by
$$\varphi(u)=\frac{1}{p}\gamma(u)-\int_{\Omega}F(z,u(z))dz\ \mbox{for all}\ u\in W^{1,p}(\Omega).$$

Evidently, $\varphi\in C^1(W^{1,p}(\Omega))$. Let $\mu>||\xi||_{\infty}$ (see hypothesis $H(\xi)$) and consider the following truncations-perturbations of the reaction term $f(z,\cdot)$:
\begin{eqnarray}
	&&g_+(z,x)=\left\{\begin{array}{ll}
		0&\mbox{if}\ x\leq 0\\
		f(z,x)+\mu x^{p-1}&\mbox{if}\ 0<x
	\end{array}\right.\label{eq6}\\
	&&g_-(z,x)=\left\{\begin{array}{ll}
		f(z,x)+\mu|x|^{p-2}x&\mbox{if}\ x<0\\
		0&\mbox{if}\ 0\leq x.
	\end{array}\right.\label{eq7}
\end{eqnarray}

Both are Carath\'eodory functions. We define
$$G_{\pm}(z,x)=\int^x_0g_{\pm}(z,s)ds$$
and consider the $C^1$-functionals $\hat{\varphi}_{\pm}:W^{1,p}(\Omega)\rightarrow\RR$ defined by
$$\hat{\varphi}_{\pm}(u)=\frac{1}{p}\gamma(u)+\frac{\mu}{p}||u||^p_p-\int_{\Omega}G_{\pm}(z,u(z))dz\ \mbox{for all}\ u\in W^{1,p}(\Omega).$$

\begin{prop}\label{prop4}
	If hypotheses $H(\xi),H(\beta),H(f)(i),(ii),(iii)$ hold, then the functionals $\hat{\varphi}_{\pm}$ satisfy the C-condition.
\end{prop}
\begin{proof}
	We shall present the proof for the functional $\hat{\varphi}_{+}$, the proof for $\hat{\varphi}_-$ is similar.
	
	So, let $\{u_n\}_{n\geq 1}\subseteq W^{1,p}(\Omega)$ be a sequence such that
	\begin{eqnarray}
		&&|\hat{\varphi}_+(u_n)|\leq M_1\ \mbox{for some}\ M_1>0,\ \mbox{all}\ n\in\NN,\label{eq8}\\
		&&(1+||u_n||)\hat{\varphi}'_+(u_n)\rightarrow 0\ \mbox{in}\ W^{1,p}(\Omega).\label{eq9}
	\end{eqnarray}
	
	From (\ref{eq9}) we have
	\begin{eqnarray}\label{eq10}
		&&\left|\left\langle A(u_n),h\right\rangle+\int_{\Omega}(\xi(z)+\mu)|u_n|^{p-2}u_nhdz+\int_{\partial\Omega}\beta(z)|u_n|^{p-2}u_nhd\sigma-\int_{\Omega}g_+(z,u_n)hdz\right|\nonumber\\
		&&\hspace{1cm}\leq\frac{\epsilon_n||h||}{1+||u_n||}\ \mbox{for all}\ h\in W^{1,p}(\Omega),\ \mbox{with}\ \epsilon_n\rightarrow 0^+.
	\end{eqnarray}
	
	In (\ref{eq10}) we choose $h=-u_n^-\in W^{1,p}(\Omega)$. Then
	\begin{eqnarray}\label{eq11}
		&&\gamma(u^-_n)+\mu||u^-_n||^p_p\leq\epsilon_n\ \mbox{for all}\ n\in\NN\ (\mbox{see (\ref{eq6})}),\nonumber\\
		&\Rightarrow&c_2||u^-_n||^p\leq\epsilon_n\ \mbox{for some}\ c_2>0,\ \mbox{all}\ n\in\NN\ (\mbox{recall}\ \mu>||\xi||_{\infty},\ \mbox{see}\ H(\beta)),\nonumber\\
		&\Rightarrow&u^-_n\rightarrow 0\ \mbox{in}\ W^{1,p}(\Omega).
	\end{eqnarray}
	
	From (\ref{eq10}), (\ref{eq11}) and (\ref{eq6}), we have
	\begin{eqnarray}\label{eq12}
		&&\left|\left\langle A(u^+_n),h\right\rangle+\int_{\Omega}\xi(z)(u^+_n)^{p-1}hdz+\int_{\partial\Omega}\beta(z)(u^+_n)^{p-1}hd\sigma-\int_{\Omega}f(z,u^+_n)hdz\right|\leq \epsilon'_n||h||\\
		&&\hspace{8cm}\mbox{for all}\ h\in W^{1,p}(\Omega),\ \mbox{with}\ \epsilon'_n\rightarrow 0.\nonumber
	\end{eqnarray}
	\begin{claim}
		$\{u^+_n\}_{n\geq 1}\subseteq W^{1,p}(\Omega)$ is bounded.
	\end{claim}
	
	We argue indirectly. So, suppose that the claim is not true. By passing to a subsequence if necessary, we may assume that $||u^+_n||\rightarrow\infty$. We set $y_n=\frac{u^+_n}{||u^+_n||}$, $n\in\NN$. Then
	$$||y_n||=1\ \mbox{and}\ y_n\geq 0\ \mbox{for all}\ n\in\NN.$$
	
	So, we may assume that
	\begin{equation}\label{eq13}
		y_n\stackrel{w}{\rightarrow} y\ \mbox{in}\ W^{1,p}(\Omega)\ \mbox{and}\ y_n\rightarrow y\ \mbox{in}\ L^p(\Omega)\ \mbox{and in}\ L^p(\partial\Omega),y\geq 0.
	\end{equation}
	
	From (\ref{eq12}) we have
	\begin{eqnarray}\label{eq14}
		&&\left|\left\langle A(y_n),h\right\rangle+\int_{\Omega}\xi(z)y^{p-1}_nhdz+\int_{\partial\Omega}\beta(z)y^{p-1}_nhd\sigma-\int_{\Omega}\frac{N_f(u^+_n)}{||u^+_n||^{p-1}}hdz\right|\leq\epsilon'_n\frac{||h||}{||u^+_n||^{p-1}}\\
		&&\hspace{7cm}\mbox{for all}\ h\in W^{1,p}(\Omega),\ \mbox{all}\ n\in\NN\nonumber.
	\end{eqnarray}
	
	From (\ref{eq5}) it is clear that
	\begin{equation}\label{eq15}
		\left\{\frac{N_f(u^+_n)}{||u^+_n||^{p-1}}\right\}_{n\geq 1}\subseteq L^{p'}(\Omega)\ \mbox{is bounded}.
	\end{equation}
	
	In (\ref{eq14}) we choose $h=y_n-y\in W^{1,p}(\Omega)$, pass to the limit as $n\rightarrow\infty$ and use (\ref{eq13}) and (\ref{eq15}). Then
	\begin{eqnarray}\label{eq16}
		&&\lim\limits_{n\rightarrow\infty}\left\langle A(y_n),y_n-y\right\rangle=0,\nonumber\\
		&\Rightarrow&y_n\rightarrow y\ \mbox{in}\ W^{1,p}(\Omega)\ \mbox{(see Proposition \ref{prop2}), hence}\ y\geq 0,\ ||y||=1.
	\end{eqnarray}
	
	Because (\ref{eq15}) and hypothesis $H(f)(ii)$, at least for a subsequence we have
	\begin{equation}\label{eq17}
		\frac{N_f(u^+_n)}{||u^+_n||^{p-1}}\stackrel{w}{\rightarrow}\hat{\lambda}_my^{p-1}\ \mbox{in}\ L^{p'}(\Omega)\ \mbox{as}\ n\rightarrow\infty
	\end{equation}
	(see Aizicovici, Papageorgiou \& Staicu \cite{1}, proof of Proposition 16). Therefore, if in (\ref{eq14}) we pass to the limit as $n\rightarrow\infty$ and use (\ref{eq16}), (\ref{eq17}), then
	\begin{eqnarray}\label{eq18}
		&&\left\langle A(y),h\right\rangle+\int_{\Omega}\xi(z)y^{p-1}hdz+\int_{\partial\Omega}\beta(z)y^{p-1}hd\sigma=\hat{\lambda}_m\int_{\Omega}y^{p-1}hdz\nonumber\\
		&&\hspace{6cm}\mbox{for all}\ h\in W^{1,p}(\Omega),\nonumber\\
		&\Rightarrow&-\Delta_py(z)+\xi(z)y(z)^{p-1}=\hat{\lambda}_my(z)^{p-1}\ \mbox{for almost all}\ z\in\Omega,\ \frac{\partial y}{\partial n_p}+\beta(z)y^{p-1}=0\\
		&&\hspace{6cm}\mbox{on}\ \partial\Omega\ (\mbox{see Papageorgiou \& R\u{a}dulescu \cite{21}}).\nonumber
	\end{eqnarray}
	
	From (\ref{eq18}) and since $y\neq 0$ (see (\ref{eq16})) and $m\geq 2$ (see hypothesis $H(f)(ii)$), it follows that $y$ must be nodal (that is, sign changing), which contradicts (\ref{eq16}). This proves the claim.
	
	From (\ref{eq11}) and the claim it follows that $\{u_n\}_{n\geq 1}\subseteq W^{1,p}(\Omega)$ is bounded and so we may assume that
	\begin{eqnarray}\label{eq19}
		u_n\stackrel{w}{\rightarrow}u\ \mbox{in}\ W^{1,p}(\Omega)\ \mbox{and}\ u_n\rightarrow u\ \mbox{in}\ L^p(\Omega)\ \mbox{and in}\ L^p(\partial\Omega).
	\end{eqnarray}
	
	Evidently,
	\begin{eqnarray}\label{eq20}
		\left\{N_{g_+}(u_n)\right\}_{n\geq 1}\subseteq L^{p'}(\Omega)\ \mbox{is bounded}
	\end{eqnarray}
	(see (\ref{eq5}), (\ref{eq6}), (\ref{eq19})). So, if in (\ref{eq10}) we choose $h=u_n-u\in W^{1,p}(\Omega)$, pass to the limit as $n\rightarrow\infty$ and use (\ref{eq19}), (\ref{eq20}), then
	\begin{eqnarray*}
		&&\lim\limits_{n\rightarrow\infty}\left\langle A(u_n),u_n-u\right\rangle=0,\\
		&\Rightarrow&u_n\rightarrow u\ \mbox{in}\ W^{1,p}(\Omega)\ (\mbox{see Proposition \ref{prop2}}),\\
		&\Rightarrow&\hat{\varphi}_+\ \mbox{satisfies the C-condition}.
	\end{eqnarray*}
	
	Similarly for the functional $\hat{\varphi}_-$, using this time (\ref{eq7}).
\end{proof}
\begin{prop}\label{prop5}
	If hypotheses $H(\xi),H(\beta),H(f)(i),(ii),(iii)$ hold, then the functional $\varphi$ satisfies the C-condition.
\end{prop}
\begin{proof}
	Let $\{u_n\}_{n\geq 1}\subseteq W^{1,p}(\Omega)$ be a sequence such that
	\begin{eqnarray}
		&&|\varphi(u_n)|\leq M_2\ \mbox{for some}\ M_2>0,\ \mbox{all}\ n\in\NN,\label{eq21}\\
		&&(1+||u_n||)\varphi'(u_n)\rightarrow 0\ \mbox{in}\ W^{1,p}(\Omega)^*\ \mbox{as}\ n\rightarrow\infty.\label{eq22}
	\end{eqnarray}
	
	From (\ref{eq22}) we have
	\begin{eqnarray}\label{eq23}
		&&\left|\left\langle A(u_n),h\right\rangle+\int_{\Omega}\xi(z)|u_n|^{p-2}u_nhdz+\int_{\partial\Omega}\beta(z)|u_n|^{p-2}u_nhd\sigma-\int_{\Omega}f(z,u_n)hdz\right|\nonumber\\
		&&\hspace{1cm}\leq\frac{\epsilon_n||h||}{1+||u_n||}\ \mbox{for all}\ h\in W^{1,p}(\Omega),\ \mbox{with}\ \epsilon_n\rightarrow 0^+
	\end{eqnarray}
	
	In (\ref{eq23}) we choose $h=u_n\in W^{1,p}(\Omega)$ and obtain
	\begin{equation}\label{eq24}
		-\gamma(u_n)+\int_{\Omega}f(z,u_n)u_ndz\leq\epsilon_n\ \mbox{for all}\ n\in\NN.
	\end{equation}
	
	On the other hand, from (\ref{eq21}) we have
	\begin{equation}\label{eq25}
		\gamma(u_n)-\int_{\Omega}pF(z,u_n)dz\leq pM_2\ \mbox{for all}\ n\in\NN.
	\end{equation}
	
	We add (\ref{eq24}) and (\ref{eq25}). Then
	\begin{equation}\label{eq26}
		\int_{\Omega}[f(z,u_n)u_n-pF(z,u_n)]dz\leq M_3\ \mbox{for some}\ M_3>0,\ \mbox{all}\ n\in\NN
	\end{equation}
	\begin{claim}
		$\{u_n\}_{n\geq 1}\subseteq W^{1,p}(\Omega)$ is bounded.
	\end{claim}
	
	We argue again by contradiction. So, suppose that $||u_n||\rightarrow\infty$. We set $y_n=\frac{u_n}{||u_n||}$ for all $n\in\NN$. Then $||y_n||=1$  and so we may assume that
	\begin{equation}\label{eq27}
		y_n\stackrel{w}{\rightarrow}y\ \mbox{in}\ W^{1,p}(\Omega)\ \mbox{and}\ y_n\rightarrow y\ \mbox{in}\ L^p(\Omega)\ \mbox{and in}\ L^p(\partial\Omega).
	\end{equation}
	
	From (\ref{eq23}) we have
	\begin{eqnarray}\label{eq28}
		&&\left|\left\langle A(y_n),h\right\rangle+\int_{\Omega}\xi(z)|y_n|^{p-2}y_nhdz+\int_{\partial\Omega}\beta(z)|y_n|^{p-2}y_nhdz-\int_{\Omega}\frac{N_f(u_n)}{||u_n||^{p-1}}hdz\right|\nonumber\\
		&&\hspace{7cm}\leq\frac{\epsilon_n}{(1+||u_n||)||u_n||^{p-1}}\ \mbox{for all}\ n\in\NN.
	\end{eqnarray}
	
	From (\ref{eq5}) it is clear that
	\begin{equation}\label{eq29}
		\left\{\frac{N_f(u_n)}{||u_n||^{p-1}}\right\}_{n\geq 1}\subseteq L^{p'}(\Omega)\ \mbox{is bounded}.
	\end{equation}
	
	So, if in (\ref{eq28}) we choose $h=y_n-y\in W^{1,p}(\Omega)$, pass to the limit as $n\rightarrow\infty$ and use (\ref{eq27}), (\ref{eq29}), then
	\begin{eqnarray}\label{eq30}
		&&\lim\limits_{n\rightarrow\infty}\left\langle A(y_n),y_n-y\right\rangle=0,\nonumber\\
		&\Rightarrow&y_n\rightarrow y\ \mbox{in}\ W^{1,p}(\Omega)\ \mbox{(see Proposition \ref{prop2}), hence}\ ||y||=1.
	\end{eqnarray}
	
	Since $y\neq 0$ (see (\ref{eq30})), for $D=\{z\in\Omega:y(z)\neq 0\}$ we have $|D|_N>0$, with $|\cdot|_N$ being the Lebesgue measure on $\RR^N$ and
	\begin{equation}\label{eq31}
		|u_n(z)|\rightarrow+\infty\ \mbox{for all}\ z\in D.
	\end{equation}
	
	Then hypothesis $H(f)(iii)$, Fatou's lemma and (\ref{eq31}) imply that
	\begin{equation}\label{eq32}
		\int_D[f(z,u_n)u_n-pF(z,u_n)]dz\rightarrow+\infty\ \mbox{as}\ n\rightarrow\infty.
	\end{equation}
	
	Also, hypotheses $H(f)(i),(iii)$ imply that we can find $c_3>0$ such that
	\begin{equation}\label{eq33}
		-c_3\leq f(z,x)x-pF(z,x)\ \mbox{for almost all}\ z\in\Omega,\ \mbox{all}\ x\in\RR.
	\end{equation}
	
	Then
	\begin{eqnarray}\label{eq34}
		&&\int_{\Omega}[f(z,u_n)u_n-pF(z,u_n)]dz\nonumber\\
		&=&\int_D[f(z,u_n)u_n-pF(z,u_n)]dz+\int_{\Omega\backslash D}[f(z,u_n)u_n-pF(z,u_n)]dz\nonumber\\
		&\geq&\int_D[f(z,u_n)u_n-pF(z,u_n)]dz-c_3|\Omega\backslash D|_N\nonumber\\
		&&\hspace{1cm}(\mbox{see (\ref{eq33}) and recall that}\ |\cdot|_N\ \mbox{is the Lebesgue measure on}\ \RR^N),\nonumber\\
		\Rightarrow&&\int_{\Omega}[f(z,u_n)u_n-pF(z,u_n)]dz\rightarrow+\infty\ \mbox{as}\ n\rightarrow\infty\ (\mbox{see (\ref{eq32})}).
	\end{eqnarray}
	
	Comparing (\ref{eq26}) and (\ref{eq34}), we reach a contradiction. This proves the claim.
	
	Because of the claim, we may assume that
	\begin{equation}\label{eq35}
		u_n\stackrel{w}{\rightarrow}u\ \mbox{in}\ W^{1,p}(\Omega)\ \mbox{and}\ u_n\rightarrow u\ \mbox{in}\ L^p(\Omega)\ \mbox{and in}\ L^p(\partial\Omega).
	\end{equation}
	
	From (\ref{eq5}) and (\ref{eq35}) we see that
	\begin{equation}\label{eq36}
		\{N_f(u_n)\}_{n\geq 1}\subseteq L^{p'}(\Omega)\ \mbox{is bounded}.
	\end{equation}
	
	So, if in (\ref{eq23}) we choose $h=u_n-u\in W^{1,p}(\Omega)$, pass to the limit as $n\rightarrow\infty$ and use (\ref{eq35}), (\ref{eq36}), then
	\begin{eqnarray*}
		&&\lim\limits_{n\rightarrow\infty}\left\langle A(u_n),u_n-u\right\rangle=0,\\
		&\Rightarrow&u_n\rightarrow u\ \mbox{in}\ W^{1,p}(\Omega),\\
		&\Rightarrow&\varphi\ \mbox{satisfies the C-condition}.
	\end{eqnarray*}
\end{proof}
\begin{prop}\label{prop6}
	If hypotheses $H(\xi), H(\beta),H(f)(iv)$ and (\ref{eq5}) hold, then $u=0$ is a local minimizer of the functionals $\hat{\varphi}_{\pm}$ and $\varphi$.
\end{prop}
\begin{proof}
	We give the proof for the functional $\hat{\varphi}_+$, the proofs for $\hat{\varphi}_-$ and $\varphi$ are similar.
	
	Hypothesis $H(f)(iv)$ implies that given $\epsilon>0$, we can find $\delta=\delta(\epsilon)>0$ such that
	\begin{equation}\label{eq37}
		F(z,x)\leq\frac{1}{p}(\vartheta(z)+\epsilon)|x|^p\ \mbox{for almost all}\ z\in\Omega,\ \mbox{all}\ |x|\leq\delta.
	\end{equation}
	
	If $r>p$, then because of (\ref{eq5}) we can find $c_4=c_4(\epsilon,r)>0$ such that
	\begin{equation}\label{eq38}
		F(z,x)\leq c_4|x|^r\ \mbox{for almost all}\ z\in\Omega,\ \mbox{all}\ |x|\geq \delta
	\end{equation}
	
	From (\ref{eq37}) and (\ref{eq38}) it follows that we can find $c_5>0$ such that
	\begin{equation}\label{eq39}
		F(z,x)\leq\frac{1}{p}(\vartheta(z)+\epsilon)|x|^p+c_5|x|^r\ \mbox{for almost all}\ z\in\Omega,\ \mbox{all}\ x\in\RR
	\end{equation}
	(recall that $\vartheta\in L^{\infty}(\Omega)$, see hypothesis $H(f)(iv)$).
	
	So, for all $u\in W^{1,p}(\Omega)$, we have
	\begin{eqnarray}\label{eq40}
		 \hat{\varphi}_+(u)&\geq&\frac{1}{p}\gamma(u)+\frac{\mu}{p}||u^-||^p_p-\frac{1}{p}\int_{\Omega}\vartheta(z)(u^+)^pdz-\frac{\epsilon}{p}||u^+||^p-c_6||u||^r\nonumber\\
		&&\hspace{3cm}\mbox{for some}\ c_6>0\ \mbox{(see (\ref{eq39}) and (\ref{eq6}))}\nonumber\\
		 &\geq&c_7||u^-||^p+\frac{1}{p}\left[\gamma(u^+)-\int_{\Omega}\vartheta(z)(u^+)^pdz-\epsilon||u^+||^p\right]-c_6||u||^r\nonumber\\
		&&\hspace{3cm}\mbox{for some}\ c_7>0\ (\mbox{recall}\ \mu>||\xi||_{\infty})\nonumber\\
		&\geq&c_7||u^-||^p+\frac{1}{p}(c_0-\epsilon)||u^+||^p-c_6||u||^r\ (\mbox{see Proposition \ref{prop3}}).
	\end{eqnarray}
	
	Choosing $\epsilon\in(0,c_0)$, from (\ref{eq40}) we have
	\begin{equation}\label{eq41}
		\hat{\varphi}_+(u)\geq c_8||u||^p-c_6||u||^r\ \mbox{for some}\ c_8>0,\ \mbox{all}\ u\in W^{1,p}(\Omega).
	\end{equation}
	
	Since $r>p$, from (\ref{eq41}) we see that we can find $\rho\in(0,1)$ small such that
	\begin{eqnarray*}
		&&\hat{\varphi}_+(u)>0=\hat{\varphi}_+(0)\ \mbox{for all}\ u\in W^{1,p}(\Omega),\ 0<||u||\leq\rho,\\
		&\Rightarrow&u=0\ \mbox{is a (strict) local minimizer of}\ \hat{\varphi}_+.
	\end{eqnarray*}

	Similarly for the functionals $\hat{\varphi}_-$ and $\varphi$.
\end{proof}

We are now ready to produce two nontrivial smooth solutions of constant sign.
\begin{prop}\label{prop7}
	If hypotheses $H(\xi),H(\beta),H(f)$ hold, then problem (\ref{eq1}) has at least two nontrivial constant sign smooth solutions
	$$u_0\in D_+\ \mbox{and}\ v_0\in-D_+.$$
\end{prop}
\begin{proof}
	Using (\ref{eq6}) and (\ref{eq7}) and the nonlinear regularity theory (see Lieberman \cite[Theorem 2]{12}), we can easily check that
	\begin{equation}\label{eq42}
		K_{\hat{\varphi}_+}\subseteq C_+\ \mbox{and}\ K_{\hat{\varphi}_-}\subseteq-C_+.
	\end{equation}
	
	So, we may assume that $u=0$ is an isolated critical point of $\hat{\varphi}_{\pm}$ or otherwise we already have whole sequence of distinct smooth positive and negative solutions of (\ref{eq1}) which as we will see in the last part of this proof, using the nonlinear strong maximum principle (see, for example, Gasinski \& Papageorgiou \cite[p. 738]{7}), belong in $D_+$ and in $-D_+$, respectively. Thus we are done.
	
	Because $u=0$ is an isolated critical point of $\hat{\varphi}_+$ and a local minimizer of $\hat{\varphi}_+$ (see Proposition \ref{prop6}), we can find $\rho\in(0,1)$ small such that
	\begin{equation}\label{eq43}
		\hat{\varphi}_+(0)=0<\inf[\hat{\varphi}_+(u):||u||=\rho]=\hat{m}^+_{\rho}
	\end{equation}
	(see Aizicovici, Papageorgiou \& Staicu \cite{1}, proof of Proposition 29).
	
	Hypothesis $H(f)(ii)$ implies that
	\begin{equation}\label{eq44}
		\hat{\varphi}_+(t\hat{u}_1)\rightarrow-\infty\ \mbox{as}\ t\rightarrow+\infty\ (\mbox{recall}\ m\geq 2)
	\end{equation}
	
	From (\ref{eq43}), (\ref{eq44}) and Proposition \ref{prop4} we see that we can apply Theorem \ref{th1} (the mountain pass theorem) and find $u_0\in W^{1,p}(\Omega)$ such that
	\begin{eqnarray}\label{eq45}
		&&u_0\in K_{\hat{\varphi}_+}\subseteq C_+\ \mbox{(see (\ref{eq42})) and}\ \hat{\varphi}_+(0)=0<\hat{m}^+_{\rho}\leq\hat{\varphi}_+(u_0)\ (\mbox{see (\ref{eq43})}),\\
		&\Rightarrow&u_0\neq 0\ \mbox{is a solution of (\ref{eq1})}.\nonumber
	\end{eqnarray}
	
	From Papageorgiou \& R\u{a}dulescu \cite{22} we know that $u_0\in L^{\infty}(\Omega)$. Let $\rho_0=||u_0||_{\infty}$. Hypotheses $H(f)(i),(iv)$ imply that we can find $\hat{\xi}_{\rho_0}>0$ such that
	\begin{equation}\label{eq46}
		f(z,x)x+\hat{\xi}_{\rho_0}|x|^p\geq 0\ \mbox{for almost all}\ z\in\Omega,\ \mbox{all}\ |x|\leq\rho_0.
	\end{equation}
	
	So, we have (see Papageorgiou \& R\u{a}dulescu \cite{21})
	\begin{eqnarray*}
		&&-\Delta_pu_0(z)+(\xi(z)+\hat{\xi}_{\rho_0})u_0(z)^{p-1}=f(z,u_0(z))+\hat{\xi}_{\rho_0}u_0(z)^{p-1}\geq 0\ \mbox{for almost all}\ z\in\Omega,\\
		&&\Rightarrow\Delta_pu_0(z)\leq(||\xi||_{\infty}+\hat{\xi}_{\rho_0})u_0(z)^{p-1}\ \mbox{for almost all}\ z\in\Omega,\\
		&&\Rightarrow u_0\in D_+\ \mbox{(by the nonlinear strong maximum principle, see \cite[p. 738]{7})}.
	\end{eqnarray*}
	
	In a similar fashion, working this time with the functional $\hat{\varphi}_-$, we produce a negative smooth solution $v_0\in-D_+$ for problem (\ref{eq1}).
\end{proof}

\section{Three Solutions Theorem}

In this section, using Morse theoretic tools (critical groups), we establish the existence of a third nontrivial smooth and thus prove the three solutions theorem for problem (\ref{eq1}) under conditions of resonance.

We start by examining the critical groups of $\varphi$ at infinity.
\begin{prop}\label{prop8}
	If hypotheses $H(\xi),H(\beta),H(f)$ hold and $K_{\varphi}$ is finite, then $C_m(\varphi,\infty)\neq 0$.
\end{prop}
\begin{proof}
	Let $\lambda\in(\hat{\lambda}_m,\hat{\lambda}_{m+1})\backslash\hat{\sigma}(p)$ and consider the $C^1$-functional $\psi:W^{1,p}(\Omega)\rightarrow\RR$ defined by
	$$\psi(u)=\frac{1}{p}\gamma(u)-\frac{\lambda}{p}||u||^p_p\ \mbox{for all}\ u\in W^{1,p}(\Omega).$$
	
	Consider the homotopy $h(t,u)=h_t(u)$ defined by
	$$h_t(u)=(1-t)\varphi(u)+t\psi(u)\ \mbox{for all}\ (t,u)\in[0,1]\times W^{1,p}(\Omega).$$
	\begin{claim}	There exist $\eta\in\RR$ and $\delta_0>0$ such that
		$$h_t(u)\leq\eta\Rightarrow(1+||u||)||(h_t)'(u)||_*\geq\delta_0\ \mbox{for all}\ t\in[0,1].$$
	\end{claim}
	
	We argue indirectly. So, suppose that the claim is not true. Because the function $(t,u)\mapsto h_t(u)$ maps bounded sets into bounded sets, we can find $\{t_n\}_{n\geq 1}\subseteq[0,1]$ and $\{u_n\}_{n\geq 1}\subseteq W^{1,p}(\Omega)$ such that
	\begin{equation}\label{eq47}
		t_n\rightarrow t,||u_n||\rightarrow\infty,h_{t_n}(u_n)\rightarrow-\infty\ \mbox{and}\ (1+||u_n||)(h_{t_n})'(u_n)\rightarrow 0\ \mbox{in}\ W^{1,p}(\Omega)^*.
	\end{equation}
	
	From the last convergence in (\ref{eq47}) we have
	\begin{eqnarray}\label{eq48}
		&&\left|\left\langle A(u_n),h\right\rangle+\int_{\Omega}\xi(z)|u_n|^{p-2}u_nhdz+\int_{\partial\Omega}\beta(z)|u_n|^{p-2}u_nhd\sigma\right.-\nonumber\\
		 &&\left.(1-t_n)\int_{\Omega}f(z,u_n)hdz-t_n\int_{\Omega}\lambda|u_n|^{p-2}u_nhdz\right|\leq\frac{\epsilon_n||h||}{1+||u_n||}\\
		&&\mbox{for all}\ h\in W^{1,p}(\Omega)\ \mbox{with}\ \epsilon_n\rightarrow 0^+.\nonumber
	\end{eqnarray}
	
	We set $y_n=\frac{u_n}{||u_n||}\ n\in\NN$. Then $||y_n||=1$ for all $n\in\NN$ and so by passing to a suitable subsequence if necessary, we may assume that
	\begin{equation}\label{eq49}
		y_n\stackrel{w}{\rightarrow}y\ \mbox{in}\ W^{1,p}(\Omega)\ \mbox{and}\ y_n\rightarrow y\ \mbox{in}\ L^p(\Omega)\ \mbox{and in}\ L^p(\partial\Omega).
	\end{equation}
	
	From (\ref{eq48}) we have
	\begin{eqnarray}\label{eq50}
		&&\left|\left\langle A(y_n),h\right\rangle+\int_{\Omega}\xi(z)|y_n|^{p-2}y_nhdz+\int_{\partial\Omega}\beta(z)|y_n|^{p-2}y_nhd\sigma-(1-t_n)\int_{\Omega}\frac{N_f(u_n)}{||u_n||^{p-1}}hdz-\right.\nonumber\\
		 &&\left.\quad t_n\int_{\Omega}\lambda|u_n|^{p-2}y_nhdz\right|\leq\frac{\epsilon_n||h||}{(1+||u_n||)||u_n||^{p-1}}\ \mbox{for all}\ n\in\NN.
	\end{eqnarray}
	
	From (\ref{eq5}) we see that
	\begin{equation}\label{eq51}
		\left\{\frac{N_f(u_n)}{||u_n||^{p-1}}\right\}_{n\geq 1}\subseteq L^{p'}(\Omega)\ \mbox{is bounded}.
	\end{equation}
	
	From hypothesis $H(f)(ii)$, (\ref{eq51}) and by passing to a subsequence if necessary, we have
	\begin{equation}\label{eq52}
		\frac{N_f(u_n)}{||u_n||^{p-1}}\stackrel{w}{\rightarrow}\hat{\lambda}_m|y|^{p-2}y\ \mbox{in}\ L^{p'}(\Omega)
	\end{equation}
	(see Aizicovici, Papageorgiou \& Staicu \cite{1}, proof of Proposition 30).
	
	In (\ref{eq50}) we choose $h=y_n-y\in W^{1,p}(\Omega)$, pass to the limit as $n\rightarrow\infty$ and use (\ref{eq49}), (\ref{eq52}). We obtain
	\begin{eqnarray}\label{eq53}	
		&&\lim\limits_{n\rightarrow\infty}\left\langle A(y_n),y_n-y\right\rangle=0,\nonumber\\
		&\Rightarrow&y_n\rightarrow y\ \mbox{in}\ W^{1,p}(\Omega)\ \mbox{(see Proposition \ref{prop2}), hence}\ ||y||=1.
	\end{eqnarray}
	
	Therefore, if in (\ref{eq50}) we pass to the limit as $n\rightarrow\infty$ and use (\ref{eq52}) and (\ref{eq53}), then
	\begin{eqnarray}\label{eq54}
		&&\left\langle A(y),h\right\rangle+\int_{\Omega}\xi(z)|y|^{p-2}yhdz+\int_{\partial\Omega}\beta(z)|y|^{p-2}yhd\sigma\nonumber\\
		&&=\int_{\Omega}[(1-t)\hat{\lambda}_m+t\lambda]\,|y|^{p-2}yhdz\ \mbox{for all}\ h\in W^{1,p}(\Omega),\nonumber\\
		&\Rightarrow&-\Delta_py(z)+\xi(z)|y(z)|^{p-2}y(z)=\lambda_t|y(z)|^{p-2}y(z)\ \mbox{for almost all}\ z\in\Omega,\\
		&&\frac{\partial y}{\partial n_{p}}+\beta(z)|y|^{p-2}y=0\ \mbox{on}\ \partial\Omega\nonumber
	\end{eqnarray}
	with $\lambda_t=(1-t)\hat{\lambda}_m+t\lambda$ (see Papageorgiou \& R\u{a}dulescu \cite{21}).
	Note that $\lambda_t\in\left[\hat{\lambda}_m,\hat{\lambda}_{m+1}\right)$.
	
	If $\lambda_t\notin\hat{\sigma}(p)$, then from (\ref{eq54}) we infer that $y=0$, a contradiction to (\ref{eq53}).
	
	If $\lambda_t\in\hat{\sigma}(p)$, then from (\ref{eq53}) we see that for $D=\{z\in\Omega:y(z)\neq 0\}$ we have $|D|_N>0$. Also, we can say that
	\begin{eqnarray}\label{eq55}
		&&|u_n(z)|\rightarrow+\infty\ \mbox{for all}\ z\in D,\nonumber\\
		&\Rightarrow&f(z,u_n(z))u_n(z)-pF(z,u_n(z))\rightarrow+\infty\ \mbox{for almost all}\ z\in D,\nonumber\\
		&\Rightarrow&\int_D[f(z,u_n)u_n-pF(z,u_n)]dz\rightarrow+\infty\\
		&&\hspace{1cm}(\mbox{by Fatou's lemma, see hypothesis}\ H(f)(iii)).\nonumber
	\end{eqnarray}
	
	Hypotheses $H(f)(i),(ii)$ imply that we can find $c_9>0$ such that
	\begin{equation}\label{eq56}
		-c_9\leq f(z,x)x-pF(z,x)\ \mbox{for almost all}\ z\in\Omega,\ \mbox{all}\ x\in\RR.
	\end{equation}
	
	We have
	\begin{eqnarray}\label{eq57}
		&&\int_{\Omega}[f(z,u_n)u_n-pF(z,u_n)]dz\nonumber\\
		&&=\int_D[f(z,u_n)u_n-pF(z,u_n)]dz+\int_{\Omega\backslash D}[f(z,u_n)u_n-pF(z,u_n)]dz\nonumber\\
		&&\geq\int_D[f(z,u_n)u_n-pF(z,u_n)]dz+c_9|\Omega\backslash D|_N\ \mbox{(see (\ref{eq56}))}\nonumber\\
		&\Rightarrow&\int_{\Omega}[f(z,u_n)u_n-pF(z,u_n)]dz\rightarrow+\infty\ \mbox{as}\ n\rightarrow\infty\ (\mbox{see (\ref{eq55})}).
	\end{eqnarray}
	
	On the other hand, from the third convergence in (\ref{eq47}), we see that we can find $n_0\in\NN$ such that
	\begin{equation}\label{eq58}
		\gamma(u_n)-(1-t_n)\int_{\Omega}pF(z,u_n)dz-t_n\lambda||u_n||^p_p\leq-1\ \mbox{for all}\ n\geq n_0.
	\end{equation}
	
	Also, if in (\ref{eq48}) we choose $h=u_n\in W^{1,p}(\Omega)$, then
	\begin{equation}\label{eq59}
		-\gamma(u_n)+(1-t_n)\int_{\Omega}f(z,u_n)u_ndz+t\lambda||u_n||^p_p\leq\epsilon_n\ \mbox{for all}\ n\in\NN.
	\end{equation}
	
	By taking $n_0\in\NN$ even bigger if necessary, we may assume that $\epsilon_n\in(0,1)$ for all $n\geq n_0$ (recall that $\epsilon_n\rightarrow 0^+$). Adding (\ref{eq58}) and (\ref{eq59}), we obtain
	\begin{equation}\label{eq60}
		(1-t_n)\int_{\Omega}[f(z,u_n)u_n-pF(z,u_n)]dz\leq 0\ \mbox{for all}\ n\geq n_0.
	\end{equation}
	
	Note that we can assume that $t_n\in\left[0,1\right)$ for all $n\geq n_0$. Indeed, if there is a subsequence $\{t_{n_k}\}_{k\geq 1}$ of $\{t_n\}_{n\geq 1}$ such that $t_{n_k}=1$ for all $k\in\NN$, then $t=1$ (see (\ref{eq47})) and from the previous argument we have
	\begin{eqnarray*}
		&&-\Delta_py(z)+\xi(z)|y(z)|^{p-2}y(z)=\lambda|y(z)|^{p-2}y(z)\ \mbox{for almost all}\ z\in\Omega,\\
		&&\frac{\partial y}{\partial n_p}+\beta(z)|y|^{p-2}y=0\ \mbox{on}\ \partial\Omega,\\
		&\Rightarrow&y=0\ (\mbox{since}\ \lambda\notin\hat{\sigma}(p)),
	\end{eqnarray*}
	which contradicts (\ref{eq53}).
	
	So, $t_n\in\left[0,1\right)$ for all $n\geq n_0$ and from (\ref{eq60}) we have
	\begin{equation}\label{eq61}
		\int_{\Omega}[f(z,u_n)u_n-pF(z,u_n)]dz\leq 0\ \mbox{for all}\ n\geq n_0.
	\end{equation}
	
	Comparing (\ref{eq57}) and (\ref{eq61}) we reach a contradiction. This proves the claim.
	
	A careful reading of the proof of the claim, reveals that the same argument with minor changes, shows that for every $t\in[0,1], h_t(\cdot)$ satisfies the C-condition. So, we can apply Theorem 5.1.21 of Chang \cite[p. 334]{3} (with $a=\eta,b\equiv+\infty$, see also Liang \& Su \cite[Proposition 3.2]{11}) and infer that
	\begin{eqnarray}\label{eq62}
		&&C_k(h_0,\infty)=C_k(h_1,\infty)\ \mbox{for all}\ k\in\NN_0,\nonumber\\
		&\Rightarrow&C_k(\varphi,\infty)=C_k(\psi,\infty)\ \mbox{for all}\ k\in\NN_0.
	\end{eqnarray}
	
	We introduce the following two sets
	\begin{eqnarray*}
		&&C_r=\{u\in W^{1,p}(\Omega):\gamma(u)<\lambda||u||^p_p,||u||=r\},\\
		&&E=\{u\in W^{1,p}(\Omega):\gamma(u)\geq\lambda||u||^p_p\}.
	\end{eqnarray*}
	
	Both are symmetric sets and $C_r\cap E=\emptyset$. The set $\partial B_r=\{u\in W^{1,p}(\Omega):||u||=r\}$ is a Banach $C^1$-manifold and so it is locally contractible. The set $C_r$ is an open subset of $\partial B_r$, so it is locally contractible, too. Evidently, the open set $W^{1,p}(\Omega)\backslash E$ is locally contractible. Since $\lambda\in(\hat{\lambda}_m,\hat{\lambda}_{m+1})\backslash\hat{\sigma}(p)$, we have
	$${\rm ind}\, C_r={\rm ind}\, (W^{1,p}(\Omega)\backslash E)=m,$$
	and recall that ind$\,(\cdot)$ denotes the Fadell-Rabinowitz  cohomological index, see \cite{6}. Moreover, from Theorem 3.6 of Cingolani \& Degiovanni \cite{4}, we know that the sets $C_r$ and $E$ homologically link in dimension $m$. So, we can apply Theorem 3.2 of Cingolani \& Degiovanni \cite{4} and infer that
	\begin{equation}\label{eq63}
		C_m(\psi,0)\neq 0.
	\end{equation}
	
	But since $\lambda\in(\hat{\lambda}_m,\hat{\lambda}_{m+1})\backslash\hat{\sigma}(p)$, we have
	\begin{eqnarray*}
		&&K_{\psi}=\{0\},\\
		&\Rightarrow&C_k(\psi,0)=C_k(\psi,\infty)\ \mbox{for all}\ k\in\NN_0\\
		&&(\mbox{see Motreanu, Motreanu \& Papageorgiou \cite{16}, Proposition 6.61, p. 160}),\\
		&\Rightarrow&C_m(\psi,\infty)\neq 0\ (\mbox{see (\ref{eq63})}),\\
		&\Rightarrow&C_m(\varphi,\infty)\neq 0\ (\mbox{see (\ref{eq62})}).
	\end{eqnarray*}
\end{proof}

In what follows we outline an alternative approach to showing that $C_m(\varphi,\infty)\neq 0$.

Note that the $p$-homogeneity of $\psi$ implies that
\begin{equation}\label{eq64}
	\psi^0\ \mbox{is contractible}
\end{equation}
(just use the radial contraction). In a similar way, we can see that
\begin{equation}\label{eq65}
	\psi^0\backslash\{0\}\ \mbox{is homotopy equivalent to}\ \psi^0\cap\partial B^{L^p}_1.
\end{equation}

Here, $\partial B^{L^p}_1=\{u\in L^p(\Omega):||u||_p=1\}$. Let $\ast\in\psi^0\backslash\{0\}$ and consider the following triple of sets
$$\{\ast\}\subseteq\psi^0\backslash\{0\}\subseteq\psi^0.$$

For this triple we consider the corresponding ``reduced" long exact sequence of singular homology groups (see Motreanu, Motreanu \& Papageorgiou \cite{16}, Proposition 6.21, p. 146).
\begin{equation}\label{eq66}
	\cdots\rightarrow H_k(\psi^0,\ast)\xrightarrow{j_*}H_k(\psi^0,\psi^0\backslash\{0\})\xrightarrow{\hat{\partial}_*}H_{k-1}(\psi^0\backslash\{0\},\ast)\rightarrow\cdots\ \quad k\in\NN,
\end{equation}
where $j_*$ is the homomorphism induced by the inclusion $(\psi^0,\ast){\overset{j}\hookrightarrow}(\psi^0,\psi^0\backslash\{0\})$ and $\hat{\partial}_*$ is the composed boundary homomorphism (see Motreanu, Motreanu \& Papageorgiou \cite{16}, Proposition 6.14, p. 143). From (\ref{eq64}) we have
\begin{equation}\label{eq67}
	H_k(\psi^0,\ast)=0\ \mbox{for all}\ k\in\NN_0
\end{equation}
(see Motreanu, Motreanu \& Papageorgiou \cite{16}, Proposition 6.24, p. 147). Also, from the exactness of (\ref{eq66}) we have
\begin{eqnarray}\label{eq68}
	&&{\rm ker}\,\hat{\partial}_*={\rm im}\,j_*=0\ (\mbox{see (\ref{eq67})}),\nonumber\\
	&\Rightarrow&\hat{\partial}_*\ \mbox{is a homomorphism onto}\ H_{k-1}(\psi^0\backslash\{0\},\ast)\ (\mbox{see (\ref{eq66})}),\nonumber\\
	&\Rightarrow&H_k(\psi^0,\psi^0\backslash\{0\})=\tilde{H}_{k-1}(\psi^0\cap\partial B^{L^p}_1)\ (\mbox{see (\ref{eq65})}).
\end{eqnarray}

Here, $\tilde{H}_{k-1}(\psi^0\cap\partial B^{L^p}_1)$ denotes the reduced homology group.

Since $\psi^0\cap\partial B^{L^p}_1=\{u\in W^{1,p}(\Omega):\gamma(u)\leq\lambda||u||^p_p,||u||_p=1\}$ and $\lambda\in(\hat{\lambda}_m,\hat{\lambda}_{m+1})\backslash\hat{\sigma}(p)$ as in Perera \cite{23}, we can show that
\begin{eqnarray*}
	&&\tilde{H}_{m-1}(\psi^0\cap\partial B^{L^p}_1)\neq 0,\\
	&\Rightarrow&H_m(\psi^0,\psi^0\backslash\{0\})\neq 0\ (\mbox{see (\ref{eq68})}),\\
	&\Rightarrow&C_m(\psi,0)\neq 0,\\
	&\Rightarrow&C_m(\psi,\infty)\neq 0\ (\mbox{recall that}\ C_m(\psi,0)=C_m(\psi,\infty)\ \mbox{since}\ K_{\psi}=\{0\}),\\
	&\Rightarrow&C_m(\varphi,\infty)\neq 0\ (\mbox{see (\ref{eq62})}).
\end{eqnarray*}

\medskip
We also compute the critical groups at infinity for the functionals $\hat{\varphi}_{\pm}$.
\begin{prop}\label{prop9}
	If hypotheses $H(\xi),H(\beta),H$ hold and $K_{\hat{\varphi}_{\pm}}$ are finite, then $C_k(\hat{\varphi}_{\pm},\infty)=0$ for all $k\in\NN_0$.
\end{prop}
\begin{proof}
	We give the proof for the functional $\hat{\varphi}_+$, the proof for $\hat{\varphi}_-$ is similar.
	
	As before, let $\lambda\in(\hat{\lambda}_m,\hat{\lambda}_{m+1})\backslash\hat{\sigma}(p)$ and consider the $C^1$-functional $\hat{\psi}_+:W^{1,p}(\Omega)\rightarrow\RR$ defined by
	$$\hat{\psi}_+(u)=\frac{1}{p}\gamma(u)+\frac{\mu}{p}||u^-||^p_p-\frac{\lambda}{p}||u^+||^p_p\ \mbox{for all}\ u\in W^{1,p}(\Omega),$$
	with $\mu>||\xi||_{\infty}$ (see hypothesis $H(\xi)$). We introduce the homotopy $\hat{h}^+(t,u)=\hat{h}^+_t(u)$ defined by
	$$\hat{h}^+_t(u)=(1-t)\hat{\varphi}_+(u)+t\hat{\psi}_+(u)\ \mbox{for all}\ (t,u)\in[0,1]\times W^{1,p}(\Omega).$$
	\begin{claim}\label{cl1}
		There exist $\hat{\eta}\in\RR$ and $\hat{\delta}_0>0$ such that
		$$\hat{h}^+_t(u)\leq\hat{\eta}\Rightarrow(1+||u||)||(\hat{h}^+_t)'(u)||_*\geq\hat{\delta}_0\ \mbox{for all}\ t\in[0,1].$$
	\end{claim}
		
	We again  argue by contradiction. So, suppose that the claim is not true. Then since $(t,u)\mapsto\hat{h}^+_t(u)$ maps bounded sets to bounded sets, we can find sequences $\{t_n\}_{n\geq 1}\subseteq[0,1]$ and $\{u_n\}_{n\geq 1}\subseteq W^{1,p}(\Omega)$ such that
	\begin{equation}\label{eq69}
		t_n\rightarrow t,||u_n||\rightarrow\infty,\hat{h}^+_{t_n}(u_n)\rightarrow-\infty\ \mbox{and}\ (1+||u_n||)(\hat{h}^+_{t_n})'(u_n)\rightarrow 0\ \mbox{in}\ W^{1,p}(\Omega)^*.
	\end{equation}
	
	From the last convergence established in (\ref{eq69}), we have
	\begin{eqnarray}\label{eq70}
		&&\left|\left\langle A(u_n),h\right\rangle+\int_{\Omega}\xi(z)|u_n|^{p-2}u_nhdz+\int_{\partial\Omega}\beta(z)|u_n|^{p-2}u_nhd\sigma-\mu\int_{\Omega}(u^-_n)^{p-1}hdz\right.\nonumber\\
		 &&\left.-(1-t_n)\int_{\Omega}f(z,u^+_n)hdz-t_n\int_{\Omega}\lambda(u^+_n)^{p-1}hdz\right|\leq\frac{\epsilon_n||h||}{1+||u_n||}\\
		&&\hspace{1cm}\mbox{for all}\ h\in W^{1,p}(\Omega),\ \mbox{with}\ \epsilon_n\rightarrow 0^+\ (\mbox{see (\ref{eq6})}).\nonumber
	\end{eqnarray}
	
	In (\ref{eq70}) we choose $h=-u^-_n\in W^{1,p}(\Omega)$. Then
	\begin{eqnarray}\label{eq71}
		&&\gamma(u^-_n)+\mu||u^-_n||^p_p\leq\epsilon_n\ \mbox{for all}\ n\in\NN,\nonumber\\
		&\Rightarrow&c_{10}||u^-_n||^p\leq\epsilon_n\ \mbox{for some}\ c_{10}>0,\ \mbox{all}\ n\in\NN\ (\mbox{since}\ \mu>||\xi||_{\infty}),\nonumber\\
		&\Rightarrow&u^-_n\rightarrow 0\ \mbox{in}\ W^{1,p}(\Omega).
	\end{eqnarray}
	
	From (\ref{eq69}) we know that $||u_n||\rightarrow\infty$. Because of (\ref{eq71}) it follows that
	$$||u^+_n||\rightarrow\infty.$$
	
	Let $y_n=\frac{u^+_n}{||u^+_n||}$, $n\in\NN$. Then $||y_n||=1$ for all $n\in\NN$. So, by passing to a subsequence if necessary, we may assume that
	\begin{equation}\label{eq72}
		y_n\stackrel{w}{\rightarrow}y\ \mbox{in}\ W^{1,p}(\Omega)\ \mbox{and}\ y_n\rightarrow y\ \mbox{in}\ L^p(\Omega)\ \mbox{and in}\ L^p(\partial\Omega),y\geq 0.
	\end{equation}
	
	From (\ref{eq70}) and (\ref{eq71}) we have
	\begin{eqnarray}\label{eq73}
		&&\left|\left\langle A(y_n),h\right\rangle+\int_{\Omega}\xi(z)|y_n|^{p-2}y_nhdz+\int_{\partial\Omega}\beta(z)|y_n|^{p-2}y_nhd\sigma-(1-t_n)\int_{\Omega}\frac{N_f(u^+_n)}{||u^+_n||^{p-1}}hdz\right.\nonumber\\
		&&\hspace{1cm}\left.-t_n\int_{\Omega}\lambda y^{p-1}_nhdz\right|\leq\epsilon'_n||h||\ \mbox{for all}\ h\in W^{1,p}(\Omega)\ \mbox{with}\ \epsilon'_n\rightarrow 0^+.
	\end{eqnarray}
	
It is clear	from (\ref{eq5})  that
	\begin{equation}\label{eq74}
		\left\{\frac{N_f(u^+_n)}{||u^+_n||^{p-1}}\right\}_{n\geq 1}\subseteq L^{p'}(\Omega)\ \mbox{is bounded}.
	\end{equation}
	
	Using (\ref{eq74}) and hypothesis $H(f)(ii)$, for at least a subsequence, we have
	\begin{equation}\label{eq75}
		\frac{N_f(u^+_n)}{||u^+_n||^{p-1}}\stackrel{w}{\rightarrow}\hat{\lambda}_my^{p-1}\ \mbox{in}\ L^{p'}(\Omega).
	\end{equation}
	
	In (\ref{eq73}) we choose $h=y_n-y\in W^{1,p}(\Omega)$, pass to the limit as $n\rightarrow\infty$ and use (\ref{eq72}) and (\ref{eq75}). Then
	\begin{eqnarray}\label{eq76}
		&&\lim\limits_{n\rightarrow\infty}\left\langle A(y_n),y_n-y\right\rangle=0,\nonumber\\
		&\Rightarrow&y_n\rightarrow y\ \mbox{in}\ W^{1,p}(\Omega),\ \mbox{hence}\ ||y||=1,y\geq 0.
	\end{eqnarray}
	
	Passing to the limit as $n\rightarrow\infty$ in (\ref{eq73}) and using (\ref{eq75}) and (\ref{eq76}), we obtain
	\begin{eqnarray}\label{eq77}
		&&\left\langle A(y),h\right\rangle+\int_{\Omega}\xi(z)|y|^{p-2}yhdz+\int_{\partial\Omega}\beta(z)|y|^{p-2}yhd\sigma=(1-t)\int_{\Omega}\hat{\lambda}_my^{p-1}hdz+\nonumber\\
		&&\hspace{1cm}t\int_{\Omega}\lambda y^{p-1}hdz\ \mbox{for all}\ h\in W^{1,p}(\Omega),\nonumber\\
		&\Rightarrow&-\Delta_py(z)+\xi(z)y(z)^{p-1}=\lambda_t y(z)^{p-1}\ \mbox{for almost all}\ z\in\Omega,\nonumber \\
		&&\hspace{1cm}\frac{\partial y}{\partial n_p}+\beta(z)y^{p-1}=0\ \mbox{on}\ \partial\Omega\ \ (\mbox{see Papageorgiou \& R\u{a}dulescu \cite{21}}),
	\end{eqnarray}
	with $\lambda_t=(1-t)\hat{\lambda}_{m}+t\lambda$.
	
	We see that $\lambda_t\in\left[\hat{\lambda}_m,\hat{\lambda}_{m+1}\right)$ and since $m\geq 2$, if $\lambda_t\in\hat{\sigma}(p)$, then it must be a nonprincipal eigenvalue and so from (\ref{eq77}) it follows that
	$$y=0\ \mbox{or}\ y\ \mbox{is nodal},$$
	both contradicting (\ref{eq76}). If $\lambda_t\notin\hat{\sigma}(p)$, then $y=0$ (see (\ref{eq77})), again a contradiction.
	
	This proves Claim \ref{cl1}.
	
	Claim \ref{cl1} says that we can find $\hat{\eta}_0\leq\hat{\eta}$ such that
	\begin{equation}\label{eq78}
		\hat{\eta}_0\ \mbox{is a regular value for all}\ \hat{h}^+_t,\ \ t\in[0,1].
	\end{equation}
	
	The above argument with minor changes also shows that
	\begin{center}
		for all $t\in[0,1],\ \hat{h}^+_t$ satisfies the C-condition.
	\end{center}
	
	We apply Theorem 5.1.21 of Chang \cite[p. 334]{3} (with $a\equiv\hat{\eta}_0$ and $b\equiv+\infty$; see also Liang \& Su \cite[Proposition 3.2]{11}) and obtain
	\begin{equation}\label{eq79}
		C_k(\hat{\varphi}_+,\infty)=C_k(\hat{\psi}_+,\infty)\ \mbox{for all}\ k\in\NN_0.
	\end{equation}
	
	Next, we consider the homotopy $\tilde{h}^+(t,u)=\tilde{h}^+_t(u)$ defined by
	$$\tilde{h}^+_t(u)=\hat{\psi}_+(u)-t\int_{\Omega}udz\ \mbox{for all}\ (t,u)\in[0,1]\times (W^{1,p}(\Omega)\backslash\{0\}).$$
	\begin{claim}\label{cl2}
		$(\tilde{h}^+_t)'(u)=0$ for all $t\in[0,1]$, all $u\in W^{1,p}(\Omega),\ u\neq 0$.
	\end{claim}
	
	We proceed by contradiction. So, suppose that we can find $t\in[0,1]$ and $u\in W^{1,p}(\Omega),\ u\neq 0$ such that
	\begin{eqnarray}\label{eq80}
		&&(\tilde{h}^+_t)'(u)=0,\nonumber\\
		&\Rightarrow&\left\langle A(u),h\right\rangle+\int_{\Omega}\xi(z)|u|^{p-2}uhdz+\int_{\partial\Omega}\beta(z)|u|^{p-2}uhd\sigma-\int_{\Omega}\mu(u^-)^{p-1}hdz\nonumber\\
		&&=\int_{\Omega}\lambda(u^+)^{p-1}hdz+t\int_{\Omega}hdz\ \mbox{for all}\ h\in W^{1,p}(\Omega).
	\end{eqnarray}
	
	In (\ref{eq80}) we choose $h=-u^-\in W^{1,p}(\Omega)$. Then
	\begin{eqnarray*}
		&&\gamma(u^-)+\mu||u^-||^2_2=t\int_{\Omega}(-u^-)dz\leq 0,\\
		&\Rightarrow&c_{11}||u^-||^2\leq 0\ \mbox{for some}\ c_{11}>0\ (\mbox{since}\ \mu>||\xi||_{\infty}),\\
		&\Rightarrow&u\geq 0,\ u\neq 0.
	\end{eqnarray*}
	
	Then equation (\ref{eq80}) becomes
	\begin{eqnarray}\label{eq81}
		&&\left\langle A(u),h\right\rangle+\int_{\Omega}\xi(z)u^{p-1}hdz+\int_{\partial\Omega}\beta(z)u^{p-1}hd\sigma=\int_{\Omega}(\lambda u^{p-1}+t)hdz\nonumber\\
		&&\hspace{5cm}\mbox{for all}\ h\in W^{1,p}(\Omega),\nonumber\\
		&\Rightarrow&-\Delta_pu(z)+\xi(z)u(z)^{p-1}=\lambda u(z)^{p-1}+t\ \mbox{for almost all}\ z\in\Omega,\\
		&&\frac{\partial u}{\partial n_p}+\beta(z)u^{p-1}=0\ \mbox{on}\ \partial\Omega\ (\mbox{see Papageorgiou \& R\u{a}dulescu \cite{21}})\nonumber.
	\end{eqnarray}
	
	As before, from (\ref{eq81}) and the nonlinear regularity theory (see \cite{12}), we have
	$$u\in C_+\backslash\{0\}.$$
	
	Also, we have
	\begin{eqnarray*}
		&&\Delta_pu(z)\leq(||\xi||_{\infty}+|\lambda|)u(z)^{p-1}\ \mbox{for almost all}\ z\in\Omega,\\
		&&\Rightarrow u\in D_+\\
		&&(\mbox{from the nonlinear strong maximum principle, see Gasinski \& Papageorgiou \cite[p. 738]{7}}).
	\end{eqnarray*}
	
	Let $v\in D_+$ and consider the function
	$$R(v,u)(z)=|Dv(z)|^p-|Du(z)|^{p-2}(Du(z),D\left(\frac{v^p}{u^{p-1}}\right)(z))_{\RR^N}.$$
	
	Using the nonlinear Picone's identity (see, for example, Motreanu, Motreanu \& Papageorgiou \cite[p. 255]{16}), we have
	\begin{eqnarray}\label{eq82}
		0&\leq&\int_{\Omega}R(v,u)dz\nonumber\\
		 &=&||Dv||^p_p-\int_{\Omega}(-\Delta_pu)\left(\frac{v^p}{u^{p-1}}\right)dz-\int_{\partial\Omega}\beta(z)v^pd\sigma\nonumber\\
		&&(\mbox{using the nonlinear Green's identity, see Gasinski \& Papageorgiou \cite[p. 211]{7}})\nonumber\\
		 &=&||Dv||^p_p+\int_{\Omega}\xi(z)v^pdz+\int_{\partial\Omega}\beta(z)v^pd\sigma-\lambda||v||^p_p-t\int_{\Omega}\frac{u^p}{v^{p-1}}dz\ (\mbox{see (\ref{eq81})})\nonumber\\
		&\leq&\gamma(v)-\lambda||v||^p_p\ (\mbox{since}\ u,v\in D_+).
	\end{eqnarray}
	
	Let $v=\hat{u}_1\in D_+$. Then from (\ref{eq82}) we have
	$$0\leq\hat{\lambda}_1-\lambda<0\ (\mbox{since}\ ||\hat{u}_1||_p=1\ \mbox{and}\ \lambda>\hat{\lambda}_m\geq\hat{\lambda}_2>\hat{\lambda}_1;\ \mbox{recall}\ m\geq 2),$$
	a contradiction.
	
	This proves Claim \ref{cl2}.
	
	The homotopy invariance property of the singular homology groups implies that for small $r>0$, we have
	\begin{eqnarray}\label{eq83}
		&&H_k((\tilde{h}^+_0)^0\cap B_r,(\tilde{h}^+_0)^0\cap B_r\backslash\{0\})=H_k((\tilde{h}^+_1)^0\cap B_r,(\tilde{h}^+_1)^0\cap B_r\backslash\{0\})\\
		&&\hspace{7cm}\mbox{for all}\ k\in\NN_0\nonumber.
	\end{eqnarray}
	
	Claim \ref{cl2} implies that
	\begin{eqnarray}\label{eq84}
		&&H_k((\tilde{h}^+_1)^0\cap B_r,(\tilde{h}^+_1)^0\cap B_r\backslash\{0\})=0\ \mbox{for all}\ k\in\NN_0\\
		&&\hspace{1cm}(\mbox{see Motreanu, Motreanu \& Papageorgiou \cite{16}, Proposition 6.61, p. 160}).\nonumber
	\end{eqnarray}
	
	Also, by definition we have
	\begin{eqnarray}\label{eq85}
		&&H_k((\tilde{h}^+_0)^0\cap B_r,(\tilde{h}^+_0) ^0\cap B_r\backslash\{0\})=C_k(\hat{\psi}_+,0)\ \mbox{for all}\ k\in\NN_0\ (\mbox{note that}\ \tilde{h}^+_0=\hat{\psi}_+)
	\end{eqnarray}
	
	From (\ref{eq83}), (\ref{eq84}), (\ref{eq85}) we infer that
	\begin{equation}\label{eq86}
		C_k(\hat{\psi}_+,0)=0\ \mbox{for all}\ k\in\NN_0.
	\end{equation}
	
	Since $\lambda\in(\hat{\lambda}_m,\hat{\lambda}_{m+1})\backslash\hat{\sigma}(p)$, we have
	\begin{eqnarray*}
		&&K_{\hat{\psi}_+}=\{0\}\\
		&\Rightarrow&C_k(\hat{\psi}_+,\infty)=C_k(\hat{\psi}_+,0)\ \mbox{for all}\ k\in\NN_0\\
		&&\hspace{1cm}(\mbox{see Motreanu, Motreanu \& Papageorgiou \cite[p. 160]{16}}),\\
		&\Rightarrow&C_k(\hat{\psi}_+,\infty)=0\ \mbox{for all}\ k\in\NN_0\ (\mbox{see (\ref{eq86})}),\\
		&\Rightarrow&C_k(\hat{\varphi}_+,0)=0\ \mbox{for all}\ k\in\NN_0\ (\mbox{see (\ref{eq79})}).
	\end{eqnarray*}
	
	Similarly for the functional $\hat{\varphi}_-$, using this time the $C^1$-functional $\hat{\psi}_-:W^{1,p}(\Omega)\rightarrow\RR$ defined by
	$$\hat{\psi}_-(u)=\frac{1}{p}\gamma(u)+\frac{\mu}{p}||u^+||^p_p-\lambda||u^-||^p_p\ \mbox{for all}\ u\in W^{1,p}(\Omega).$$
\end{proof}

This proposition permits the exact computation of the critical groups of $\varphi$ at the two nontrivial constant sign smooth solutions $u_0\in D_1$ and $v_0\in-D_+$ (see Proposition \ref{prop7}).
\begin{prop}\label{prop10}
	If hypotheses $H(\xi),H(\beta),H(f)$ hold, $u_0\in D_+$ and $v_0\in -D_+$ from Proposition \ref{prop7} are the only nontrivial constant sign smooth solutions of (\ref{eq1}) and $K_{\varphi}$ is finite, then $C_k(\varphi,u_0)=C_k(\varphi,v_0)=\delta_{k,1}\ZZ$ for all $k\in\NN_0$.
\end{prop}
\begin{proof}
	We give the proof for the solution $u_0\in D_+$, the proof for $v_0\in-D_+$ is similar.
	
	From (\ref{eq42}) and the hypothesis of the proposition, we have
	$$K_{\hat{\varphi}_+}=\{0,u_0\}.$$
	
	Let $\eta<0<a<\hat{m}^+_{\rho}$ (see (\ref{eq43})). We consider the following triple of sets
	$$\hat{\varphi}^{\eta}_+\subseteq\hat{\varphi}^a_+\subseteq W^{1,p}(\Omega).$$
	
	For this triple, we consider the corresponding long exact sequence of singular homology groups (see Motreanu, Motreanu \& Papageorgiou \cite[Proposition 6.14, p. 143]{16}). We have
	\begin{equation}\label{eq87}
		\cdots\rightarrow H_k(W^{1,p}(\Omega),\hat{\varphi}^{\eta}_+)\xrightarrow{j_*}H_k(W^{1,p}(\Omega),\hat{\varphi}^a_+)\xrightarrow{\hat{\partial}_*}H_{k-1}(\hat{\varphi}^a_+,\hat{\varphi}^{\eta}_+)\rightarrow\cdots\ \mbox{for all}\ k\in\NN.
	\end{equation}
	
	Here, $j_*$ is the group homomorphism induced by the inclusion
	$$(W^{1,p}(\Omega),\hat{\varphi}^{\eta}_+){\overset{j}\hookrightarrow}(W^{1,p}(\Omega),\hat{\varphi}^a_+)$$
	and $\hat{\partial}_*$ is the composed boundary homomorphism (see \cite{16}). By the rank theorem, we have
	\begin{eqnarray}\label{eq88}
		{\rm rank}\, H_k(W^{1,p}(\Omega),\hat{\varphi}^a_+)&=&{\rm rank}\, {\rm ker}\,\hat{\partial}_*+{\rm rank}\, {\rm im}\,\hat{\partial}_*\nonumber\\
		&=&{\rm rank}\, {\rm im}\, j_*+{\rm rank}\, {\rm im}\,\hat{\partial}_*\\
		&&(\mbox{from the exactness of (\ref{eq87})}).\nonumber
	\end{eqnarray}
	
	From the choice of $\eta$ and since $K_{\hat{\varphi}_+}=\{0,u_0\}$, we have
	\begin{equation}\label{eq89}
		H_k(W^{1,p}(\Omega),\hat{\varphi}^{\eta}_+)=C_k(\hat{\varphi}_+,u_0)\ \mbox{for all}\ k\in\NN_0.
	\end{equation}
	
	Also because $a\in(0,\hat{m}^+_{\rho})$ and since $K_{\hat{\varphi}_+}=\{0,u_0\}$, we have
	\begin{eqnarray}\label{eq90}
		&&H_k(W^{1,p}(\Omega),\hat{\varphi}^a_+)=C_k(\hat{\varphi}_+,u_0)\ \mbox{for all}\ k\in\NN_0\\
		&&\hspace{1cm}(\mbox{see Motreanu, Motreanu \& Papageorgiou \cite{16}, Lemma 6.55, p. 157})\nonumber.
	\end{eqnarray}
	
	Finally, because $\eta<0<a<\hat{m}^+_{\rho}$ and since $K=\{0,u_0\}$, we have
	\begin{eqnarray}\label{eq91}
		 &&H_{k-1}(\hat{\varphi}^a_+,\hat{\varphi}^{\eta}_+)=C_{k-1}(\hat{\varphi}_+,0)=\delta_{k-1,0}\ZZ=\delta_{k,1}\ZZ\ \mbox{for all}\ k\in\NN_0\\
		&&\hspace{8cm}(\mbox{see Proposition \ref{prop6}}).\nonumber
	\end{eqnarray}
	
	Returning to (\ref{eq88}) and using (\ref{eq89}), (\ref{eq90}), (\ref{eq91}), we see that
	\begin{equation}\label{eq92}
		{\rm rank}\, C_1(\hat{\varphi}_+,u_0)\leq 1.
	\end{equation}
	
	On the other hand, $u_0\in D_+$ is a critical point of mountain pass type (see the proof of Proposition \ref{prop7}). So, we have
	\begin{equation}\label{eq93}
		C_1(\varphi, u_0)\neq 0
	\end{equation}
	(see Corollary 6.81, p. 168 of Motreanu, Motreanu \& Papageorgiou \cite{16}). From (\ref{eq92}), (\ref{eq93}) and since in (\ref{eq87}) the part $k\geq 2$ is trivial (see (\ref{eq89}), (\ref{eq90})), we infer that
	\begin{equation}\label{eq94}
		C_k(\hat{\varphi}_+,u_0)=\delta_{k,1}\ZZ\ \mbox{for all}\ k\in\NN_0.
	\end{equation}
	
	We consider the homotopy $h^+(t,u)=h^+_t(u)$ defined by
	$$h^+_t(u)=(1-t)\varphi(u)+t\hat{\varphi}_+(u)\ \mbox{for all}\ (t,u)\in[0,1]\times W^{1,p}(\Omega).$$
	
	Suppose we could find $\{t_n\}_{n\geq 1}\subseteq[0,1]$ and $\{u_n\}_{n\geq 1}\subseteq W^{1,p}(\Omega)$ such that
	\begin{equation}\label{eq95}
		t_n\rightarrow t\in[0,1],u_n\rightarrow u_0\ \mbox{in}\ W^{1,p}(\Omega)\ \mbox{and}\ (h^+_{t_n})'(u_n)=0\ \mbox{for all}\ n\in\NN.
	\end{equation}
	
	From the equality in (\ref{eq95}), we have
	\begin{eqnarray}\label{eq96}
		&&\left\langle A(u_n),h\right\rangle+\int_{\Omega}\xi(z)|u_n|^{p-2}u_nhdz+\int_{\partial\Omega}\beta(z)|u_n|^{p-2}hd\sigma-t_n\int_{\Omega}(u^-_n)^{p-1}hdz\nonumber\\
		&&=(1-t_n)\int_{\Omega}f(z,u_n)hdz+t_n\int_{\Omega}f(z,u^+_n)hdz\ \mbox{for all}\ h\in W^{1,p}(\Omega),\nonumber\\
		 &\Rightarrow&-\Delta_pu_n(z)+\xi(z)|u_n(z)|^{p-2}u_n(z)=(1-t_n)f(z,u_n(z))+t_nf(z,u_n^+(z))+t_n(u^-_n)^{p-1}\nonumber\\
		&&\mbox{for almost all}\ z\in\Omega,\ \frac{\partial u_n}{\partial n_p}+\beta(z)|u_n|^{p-2}u_n=0\ \mbox{on}\ \partial\Omega,\ \mbox{for all}\ n\in\NN\\
		&&(\mbox{see Papageorgiou \& R\u{a}dulescu \cite{21}}).\nonumber
	\end{eqnarray}
	
	From (\ref{eq96}) and Papageorgiou \& R\u{a}dulescu \cite{22}, we know that we can find $M_4>0$ such that
	$$||u_n||_{\infty}\leq M_4\ \mbox{for all}\ n\in\NN.$$
	
	Using Theorem 2 of Lieberman \cite{12}, we can find $\alpha\in(0,1)$ and $M_5>0$ such that
	\begin{equation}\label{eq97}
		u_n\in C^{1,\alpha}(\overline{\Omega})\ \mbox{and}\ ||u_n||_{C^{1,\alpha}(\overline{\Omega})}\leq M_5\ \mbox{for all}\ n\in\NN.
	\end{equation}
	
	Exploiting the compact embedding of $C^{1,\alpha}(\overline{\Omega})$ into $C^1(\overline{\Omega})$, we infer from (\ref{eq95}) and (\ref{eq97}) that
	$$u_n\rightarrow u_0\in D_+\ \mbox{in}\ C^1(\overline{\Omega})\ \mbox{as}\ n\rightarrow\infty .$$
	
	Since $D_+\subseteq C^1(\overline{\Omega})$ is open, we can find $n_0\in\NN$ such that
	$$u_n\in D_+\ \mbox{for all}\ n\geq n_0.$$
	
	But note that $\varphi|_{C_+}=\hat{\varphi}_+|_{C_+}$ (see (\ref{eq6})). So, we can see that $\{u_n\}_{n\geq 1}\subseteq K_{\varphi}$ and this contradicts our hypothesis that the critical set $K_{\varphi}$ is finite. Therefore (\ref{eq95}) cannot happen and so we can use Theorem 5.2 of Corvellec \& Hantoute \cite{5} (homotopy invariance of critical groups) and have
	\begin{eqnarray*}
		&&C_k(h^+_0,u_0)=C_k(h^+_1,u_0)\ \mbox{for all}\ k\in\NN_0,\\
		&\Rightarrow&C_k(\varphi,u_0)=C_k(\hat{\varphi}_+,u_0)\ \mbox{for all}\ k\in\NN_0,\\
		&\Rightarrow&C_k(\varphi,u_0)=\delta_{k,1}\ZZ\ \mbox{for all}\ k\in\NN_0\ (\mbox{see (\ref{eq94})}).
	\end{eqnarray*}
	
	Similarly for the negative solution $v_0\in-D_+$ using this time the functional $\hat{\varphi}_-$ and (\ref{eq7}).
\end{proof}

We are now ready for the complete multiplicity theorem (three solutions theorem) for problem (\ref{eq1}).
\begin{theorem}\label{th11}
	If hypotheses $H(\xi),H(\beta),H(f)$ hold, then problem (\ref{eq1}) has at least three nontrivial smooth solutions
	$$u_0\in D_+,v_0\in-D_+\ \mbox{and}\ y_0\in C^1(\overline{\Omega})\backslash\{0\}.$$
\end{theorem}
\begin{proof}
	From Proposition \ref{prop7}, we already have two nontrivial constant sign smooth solutions
	$$u_0\in D_+\ \mbox{and}\ v_0\in-D_+.$$
	
	Suppose that $K_{\varphi}$ is finite. Otherwise we already have an infinity of nontrivial solutions in addition to $u_0,v_0$ which belong in $C^1(\overline{\Omega})$ (by the nonlinear regularity theory, see \cite{12}) and so we are done.
	
	From Proposition \ref{prop10}, we have
	\begin{eqnarray}\label{eq98}
		C_k(\varphi,u_0)=C_k(\varphi,v_0)=\delta_{k,1}\ZZ\ \mbox{for all}\ k\in\NN_0.
	\end{eqnarray}
	
	From Proposition \ref{prop6}, we have
	\begin{equation}\label{eq99}
		C_k(\varphi,0)=\delta_{k,0}\ZZ\ \mbox{for all}\ k\in\NN_0.
	\end{equation}
	
	According to Proposition \ref{prop8}, $C_m(\varphi,\infty)\neq 0$. So, there exists $y_0\in W^{1,p}(\Omega)$ such that
	\begin{equation}\label{eq100}
		y_0\in K_{\varphi}\ \mbox{and}\ C_m(\varphi,y_0)\neq 0\ (\mbox{see Section 2}).
	\end{equation}
	
	Since $m\geq 2$, we infer from (\ref{eq98}), (\ref{eq99}), (\ref{eq100}) that
	$$y_0\notin\{0,u_0,v_0\}.$$
	
	Therefore $y_0$ is the third nontrivial solution of (\ref{eq1}) (see (\ref{eq100})) and by the nonlinear regularity theory (see \cite{12}), we have $y_0\in C^1(\overline{\Omega})$.
\end{proof}

\medskip
{\bf Acknowledgments}. This research  was supported in part by  the  Slovenian  Research  Agency
grants P1-0292, J1-7025, and J1-6721. V. R\u adulescu was supported by a grant of the Romanian National
Authority for Scientific Research and Innovation, CNCS-UEFISCDI, project number PN-III-P4-ID-PCE-2016-0130.

\end{document}